\documentclass[a4paper,10pt,reqno]{amsart}

\usepackage[utf8]{inputenc}
\usepackage[T1]{fontenc}
\usepackage{amsthm}
\usepackage{amsmath}
\usepackage{amssymb}
\usepackage[inline]{enumitem}
\usepackage{comment}
\PassOptionsToPackage{hyphens}{url}\usepackage{hyperref}
\usepackage{fancyhdr}
\usepackage{mathrsfs}
\usepackage{stmaryrd}
\usepackage[normalem]{ulem}
\usepackage{xcolor}

\setlist[description]{%
  itemsep=0.05cm,               
  font={\normalfont\textsc}, 
 leftmargin=\parindent,
 labelindent=\parindent
}

\theoremstyle{definition}

\newtheorem{theorem}{Theorem}[section]
\newtheorem{lemma}[theorem]{Lemma}
\newtheorem{proposition}[theorem]{Proposition}
\newtheorem{corollary}[theorem]{Corollary}

\theoremstyle{definition}

\newtheorem{remark}[theorem]{Remark}
\newtheorem{question}{Question}

\newtheorem{remarks}[theorem]{Remarks}

\newtheorem*{conjecture*}{Conjecture}
\newtheorem*{problem*}{Problem}
\newtheorem*{question*}{Question}

\definecolor{blue-url}{RGB}{0,0,100}
\definecolor{red-url}{RGB}{100,0,0}
\definecolor{green-url}{RGB}{0,100,0}
\definecolor{light-yellow}{RGB}{255,255,128}
\definecolor{light-blue}{RGB}{193,255,255}
\definecolor{light-red}{RGB}{239,83,80}

\hypersetup{
	pdftitle={Power monoids and automorphisms},
	pdfauthor={Tringali and Yan},
	pdfmenubar=false,
	pdffitwindow=true,
	pdfstartview=FitH,
	colorlinks=true,
	linkcolor=blue-url,
	citecolor=green-url,
	urlcolor=red-url
}

\renewcommand{\emptyset}{\varnothing}
\renewcommand{\setminus}{\smallsetminus}
\renewcommand{\,}{\kern 0.1em}

\providecommand\llb{\llbracket}
\providecommand\rrb{\rrbracket}

\newcommand{\evid}[1]{\textsf{#1}}
\newcommand{\fin}{\mathrm{fin}}
\newcommand{\aut}{\mathrm{Aut}}
\newcommand{\fdim}{\mathrm{b.dim}}
\newcommand{\maxgap}{\overline{\mathrm{gap}}}
	
{\newline\vspace{\abovedisplayskip}\hbox to \textwidth\bgroup\hss$\displaystyle}
{$\hss\egroup\vspace{\belowdisplayskip}}

\makeatletter
\DeclareFontFamily{OMX}{MnSymbolE}{}
\DeclareSymbolFont{MnLargeSymbols}{OMX}{MnSymbolE}{m}{n}
\SetSymbolFont{MnLargeSymbols}{bold}{OMX}{MnSymbolE}{b}{n}
\DeclareFontShape{OMX}{MnSymbolE}{m}{n}{
	<-6>  MnSymbolE5
	<6-7>  MnSymbolE6
	<7-8>  MnSymbolE7
	<8-9>  MnSymbolE8
	<9-10> MnSymbolE9
	<10-12> MnSymbolE10
	<12->   MnSymbolE12
}{}
\DeclareFontShape{OMX}{MnSymbolE}{b}{n}{
	<-6>  MnSymbolE-Bold5
	<6-7>  MnSymbolE-Bold6
	<7-8>  MnSymbolE-Bold7
	<8-9>  MnSymbolE-Bold8
	<9-10> MnSymbolE-Bold9
	<10-12> MnSymbolE-Bold10
	<12->   MnSymbolE-Bold12
}{}

\let\llangle\@undefined
\let\rrangle\@undefined
\DeclareMathDelimiter{\llangle}{\mathopen}%
{MnLargeSymbols}{'164}{MnLargeSymbols}{'164}
\DeclareMathDelimiter{\rrangle}{\mathclose}%
{MnLargeSymbols}{'171}{MnLargeSymbols}{'171}
\makeatother

\begin{document}
\title{On Power Monoids and their Automorphisms}
\author{Salvatore Tringali}
\address{(S.~Tringali) School of Mathematical Sciences, Hebei Normal University | Shijiazhuang, Hebei province, 050024 China}
\email{salvo.tringali@gmail.com}
\urladdr{http://imsc.uni-graz.at/tringali}
\author{Weihao Yan}
\address{(W.~Yan) School of Mathematical Sciences, Hebei Normal University | Shijiazhuang, Hebei province, 050024 China}
\email{weihao.yan.hebnu@outlook.com}



\begin{abstract}
Endowed with the binary operation of set addition, the family $\mathcal{P}_{{\rm fin},0}(\mathbb{N})$ of all finite subsets of $\mathbb{N}$ containing $0$ forms a monoid, with the singleton $\{0\}$ as its neutral element.

We show that the only non-trivial automorphism of $\mathcal P_{{\rm fin},0}(\mathbb N)$ is the involution $X \mapsto \max X - X$. The proof leverages ideas from additive number theory and proceeds through an unconventional induction on what we call the boxing dimension of a finite set of integers, that is, the smallest number of (discrete) intervals whose union is the set itself.
\end{abstract}

\maketitle

\thispagestyle{empty}

\section{Introduction}
\label{sec:intro}

Endowed with the binary op\-er\-a\-tion of setwise multiplication 
$(X, Y) \mapsto \{xy \colon x \in X,\, y \in Y\}$,
the finite subsets of a (multiplicatively written) monoid $H$ containing the neutral element $1_H$ of $H$ form a (multiplicatively written) monoid, hereby denoted by $\mathcal P_{\fin,1}(H)$ and called the \evid{reduced} (\evid{finitary}) \evid{power monoid} of $H$. In particular, the neutral element of $\mathcal P_{\fin,1}(H)$ is the singleton $\{1_H\}$. (We refer to Howie's monograph \cite{Ho95} for generalities on monoids.)  

Reduced power monoids are the most basic items in a complex hierarchy of ``highly non-can\-cel\-la\-tive'' monoids, generically referred to as \textsf{power monoids}. These structures have been first systematically studied by Tamura and Shafer in the late 1960s \cite{Tam-Sha1967}, and more recently investigated in a series of papers by Fan and Tringali \cite{Fa-Tr18}, Antoniou and Tringali \cite{An-Tr18}, Bienvenu and Geroldinger \cite{Bie-Ger-22}, and Tringali and Yan \cite{Tri-Yan-23(a)}.

There are many reasons why power monoids are interesting. First, they provide a natural algebraic framework for a number of important problems in additive combinatorics and related fields, like S\'{a}rk\"ozy's conjecture \cite[Conjecture 1.6]{Sark08} on the ``additive irreducibility'' of the set of quadratic residues modulo $p$ for all but finitely many primes $p \in \mathbb N$. 
Second, the arithmetic of power monoids --- with emphasis on ques\-tions concerned with the possibility or impossibility of factoring certain sets as a product of other sets that, in a sense, cannot be ``broken down into smaller pieces'' --- is a rich topic in itself and has been pivotal to an ongoing radical transformation \cite{Tr20(c), Co-Tr-21(a), Co-Tr-22(a), Tr21(b), Co-Tr-22(b)} of the ``abstract theory of factorization'', a subfield of algebra so far largely confined to commutative domains and cancellative commutative monoids \cite{Ger-Hal-06, Ge-Zh-20a, Got-And-2022}. Third, power monoids (and, more generally, power semigroups) have long been known to play a central role in the theory of automata and formal languages \cite{Alm02, Pin1995}.

In the present paper, we address the problem of determining the automorphism group of $\mathcal P_{\fin,1}(H)$. (Unless specified otherwise, a mor\-phism will always mean a \emph{monoid hom\-o\-mor\-phism}, cf.~Remark \ref{rem:autos}\ref{rem:autos(1)}.) 
Apart from being a natural question (the automorphisms of an object in a certain category are a measure of the ``symmetries'' of the object itself and their investigation is a fundamental problem in many areas of mathematics), a better understanding of the algebraic properties of power monoids will hopefully lead to a better understanding of other aspects of their theory, such as the following question, which arose from a conjecture of Bienvenu and Geroldinger \cite[Conjecture 4.7]{Bie-Ger-22} and was settled by the authors in \cite{Tri-Yan-23(a)}:

\begin{question}\label{quest:2}
Given a class $\mathcal C$ of monoids, prove or disprove that $\mathcal P_{\fin,1}(H)$ is isomorphic to $\mathcal P_{\fin,1}(K)$, for some $H, K \in \mathcal C$, if and only if $H$ is isomorphic to $K$.
\end{question}

To set the stage, let $\aut(M)$ be the automorphism group of a monoid $M$. It is not difficult to show (see \cite[Remark 1.1]{Tri-Yan-23(a)}) that, for each $f \in \aut(H)$, the function 
$$
\mathcal P_{\fin,1}(H) \to \mathcal P_{\fin,1}(H) \colon X \mapsto f[X] := \{f(x) \colon x \in X\}
$$
is an automorphism of $\mathcal P_{\fin,1}(H)$, henceforth referred to as the \evid{augmentation} of $f$. Accordingly, we call an auto\-mor\-phism of $\mathcal P_{\fin,1}(H)$ \evid{inner} if it is the augmentation of an automorphism of $H$.

So, we have a well-defined map $\Phi \colon \mathrm{Aut}(H) \to \mathrm{Aut}(\mathcal P_{\fin,1}(H))$ sending an automorphism of $H$ to its aug\-men\-ta\-tion.
In addition, it is easily checked that
$$
\Phi(g \circ f)(X) = g \circ f[X] = g[f[X]] = g[\Phi(f)(X)] = \Phi(g)(\Phi(f)(X)) = \Phi(g) \circ \Phi(f)(X),
$$
 for all $f, g \in \aut(H)$ and $X \in \mathcal P_{\fin,1}(H)$.
If, on the other hand, $f(x) \ne g(x)$ for some $x \in H$ and we take $X := \{1_H, x\}$, then $\Phi(f)(X) = \{1_H, f(x)\} \ne \{1_H, g(x)\} = \Phi(g)(X)$ and hence $\Phi(f) \ne \Phi(g)$.

In other words, $\Phi$ is an injective (group) homomorphism from $\mathrm{Aut}(H)$ to $\mathrm{Aut}(\mathcal P_{\fin,1}(H))$. It is therefore natural to ask if $\Phi$ is also surjective and hence an iso\-mor\-phism, which is tantamount to saying that every automorphism of $\mathcal P_{\fin,1}(H)$ is inner. Unsurprisingly, this is not generally the case and determining $\mathrm{Aut}(\mathcal P_{\fin,1}(H))$ is usually much harder than determining $\mathrm{Aut}(H)$, as we are going to demonstrate in the basic case where $H$ is the additive monoid $(\mathbb N, +)$ of non-negative integers. 

More precisely, let us denote by $\mathcal P_{\fin,0}(\mathbb N)$ the reduced power monoid of $(\mathbb N, +)$. In contrast to what we have done so far with $\mathcal P_{\fin,1}(H)$, we will write $\mathcal P_{\fin,0}(\mathbb N)$ additively, so that the operation on $\mathcal P_{\fin,0}(\mathbb N)$ maps a pair $(X, Y)$ of finite subsets of $\mathbb N$ containing $0$ to their sumset 
$$
X + Y := \{x+y \colon x \in X, \, y \in Y\}.
$$
We establish in Theorem \ref{thm:characterization} that an endomorphism of $\mathcal P_{{\rm fin},0}(\mathbb N)$ is an automorphism if and only if it is surjective, if and only if it is injective and sends $\{0, 1\}$ to $\{0, 1\}$. Then, in Theorem \ref{thm:main}, we prove that the only non-trivial automorphism of $\mathcal P_{{\rm fin},0}(\mathbb N)$ is the involution $X \mapsto \allowbreak \max X - \allowbreak X$. Since the automorphism group of $(\mathbb N, +)$ is trivial, this shows in particular that the automorphisms of $\mathcal P_{\fin,0}(\mathbb N)$ are not all inner. The proof of the latter result proceeds through an unconventional induction on
what we call the \evid{boxing dimension} of a finite set of integers, i.e., the smallest number of intervals whose union is the set itself.

We have striven to make our arguments as elementary and self-contained as possible, up to a classical theorem of Nathanson \cite[Theorem 1.1]{Nat1996}, sometimes called the \emph{Fundamental Theorem of Additive Number Theory} \cite[p.~3]{Nat2010}, that enters the proof of Lemma \ref{lem:FTAC-decomposition}. We propose directions for further research in Sect.~\ref{sec:conclusions}.

\subsection*{Notation and terminology.} We denote by $\mathbb N$ the (set of) non-negative integers, by $\mathbb N^+$ the positive integers, and by $|\cdot|$ the cardinality of a set. If $a, b \in \mathbb Z \cup \{\pm \infty\}$, we let $\llb a, b \rrb := \{x \in \allowbreak \mathbb Z \colon \allowbreak a \le x \le b\}$ be the (\evid{discrete}) \evid{interval} from $a$ to $b$. Given $h \in \mathbb N$ and $k \in \mathbb Z$, we use $hX$ for the $h$-fold sum and $k \times X := \allowbreak \{kx \colon x \in X\}$ for the \evid{$k$-dilate}
of a set $X \subseteq \mathbb Z$, resp. Moreover, we put $-X := -1 \times X$ and $X^+ := X \cap \mathbb N^+$. 
Further notation and ter\-mi\-nol\-o\-gy, if not explained, are standard or should be clear from the context.

\section{Preliminaries}
\label{sec:preliminaries}

In this section, we work out a variety of preliminaries that will be used in Sect.~\ref{sect:3} to prove Theorem \ref{thm:main} (that is, the main result of the paper). 
We start with a short list of remarks.

\begin{remarks}
\label{rem:autos}
\begin{enumerate*}[label=\textup{(\arabic{*})}, mode=unboxed]
\item\label{rem:autos(1)} Let $H$ and $K$ be (multiplicatively written) monoids. A \emph{semigroup} homomorphism from $H$ to $K$ is, by definition, a function $f$ from $H$ to $K$ that satisfies the functional equation:
\begin{equation}\label{cauchy-functional-eq}
f(xy) = f(x) f(y), \qquad
\text{for all }
x, y \in H.
\end{equation}
A solution to Eq.~\eqref{cauchy-functional-eq} will send the neutral element $1_H$ of $H$ to an i\-dem\-po\-tent of $K$. So, if the only i\-dem\-po\-tent of $K$ is the neutral element $1_K$, then any semigroup ho\-mo\-mor\-phism from $H$ to $K$ is also a \emph{monoid} ho\-mo\-mor\-phism (namely, it maps $1_H$ to $1_K$). This happens, in particular, when $K$ is the reduced power monoid of $(\mathbb N, +)$. In fact, if $2X = X$ for some $X \in \mathcal P_{\fin,0}(\mathbb N)$, then $2 \max X = \max X$ and hence $\max X = 0$, which forces $X$ to be the singleton $\{0\}$ (i.e., the neutral element of $\mathcal P_{\fin,0}(\mathbb N)$). 
\end{enumerate*}

\vskip 0.05cm

\begin{enumerate*}[label=\textup{(\arabic{*})}, mode=unboxed, resume]
\item\label{rem:autos(2)}
The only automorphism of $(\mathbb N, +)$ is the identity function on $\mathbb N$, because $(\mathbb N, +)$ is a cyclic monoid generated by $1$ and isomorphisms map generators to generators. Yet, the automorphism group of $\mathcal P_{\fin,0}(\mathbb N)$ is non-trivial, for it contains the \evid{reversion map} 
$$
\textrm{rev} \colon \mathcal P_{\fin,0}(\mathbb N) \to \mathcal P_{\fin,0}(\mathbb N) \colon X \mapsto \max X - X.
$$
Indeed, $\mathrm{rev}(\cdot)$ is an involution (or self-inverse function); and it is clear that, for all $X, Y \in \mathcal P_{\fin,0}(\mathbb N)$,
$$
\mathrm{rev}(X+Y) = \max(X + Y) - (X + \allowbreak Y) = \allowbreak (\max X - \allowbreak X) + \allowbreak (\max Y - \allowbreak Y) = \mathrm{rev}(X) + \allowbreak \mathrm{rev}(Y).
$$
Since 
$\mathrm{rev}(\{0, 1, 3\}) = \{0, \allowbreak 2, \allowbreak 3\}$, we thus get from item \ref{rem:autos(1)} that $\mathrm{rev}(\cdot)$ is a non-trivial automorphism of $\mathcal P_{\fin,0}(\mathbb N)$. We will see in Theorem \ref{thm:main} that $\mathcal P_{\fin,0}(\mathbb N)$ has no other non-trivial automorphism.
\end{enumerate*}

\vskip 0.05cm

\begin{enumerate*}[label=\textup{(\arabic{*})}, mode=unboxed, resume]
\item\label{rem:autos(3)} Let $H$ be a (multiplicatively written) monoid. If $x$ and $y$ are fixed points of an endomorphism $f$ of $H$, then so is, of course, the product $xy$. Thus, the set of fixed points of $f$ is a submonoid of $H$.
\end{enumerate*}
\end{remarks}

We continue with a series of lemmas that will be crucial in the subsequent analysis. The first of these lemmas is essentially a generalization of \cite[Lemma 2.2]{Tri-Yan-23(a)}.

\begin{lemma}\label{lem:FTAC-decomposition}
If $A \in \mathcal P_{\fin,0}(\mathbb N)$ and $\{0, \max A\} \subseteq B \subseteq A$, then $(k+1)A = kA + B$ for all large $k \in \mathbb N$.
\end{lemma}

\begin{proof}
Set $a := \max A$ and $q := \gcd A$. Since $kA + \{0, a\} \subseteq kA + B \subseteq (k+1)A$ for every $k \in \mathbb N$, it suffices to show that $(k+1)A \subseteq kA + \{0, a\}$ when $k$ is large enough. We may also suppose $A \ne \{0\}$, or else $kA = \allowbreak \{0\}$ for all $k \in \mathbb N$ and the conclusion is obvious. Accordingly, $q$ is a positive integer and there is no loss of generality in as\-sum\-ing that $q = 1$, because taking $\hat{A} := \{x/q \colon x \in A\} \subseteq \mathbb N$ and $\hat{a} := \max \hat{A}$ yields (i) $\gcd \hat{A} = \allowbreak 1$ and (ii) $(k+1)A \subseteq kA + \{0, a\}$, for some $k \in \mathbb N$, if and only if $(k+1)\hat{A} \subseteq k\hat{A} + \{0,\hat{a}\}$.

Consequently, we gather from the Fundamental Theorem of Additive Number Theory \cite[Theorem 1.1]{Nat1996} that there exist $c, d, k_0 \in \mathbb N$ and sets $C \subseteq \llb 0, c-2 \rrb$ and $D \subseteq \llb 0, d-2\rrb$ such that 
\begin{equation}\label{equ:FTAC-applied}
kA = C \cup \llb c, ka - d \, \rrb \cup (ka - D),
\qquad
\text{for each integer } k \ge k_0.
\end{equation}
Now, fix an integer $h \ge \max \{k_0,  1+(c+d)/a\}$, and let $x \in (h + 1)A$. 
Since $h \ge k_0$, Eq.~\eqref{equ:FTAC-applied} holds for both $k = h$ and $k = h+1$. So, either $x \in C \cup \llb c, ha - d - 1 \rrb$, and it is then clear that $x \in hA$; or $x - \allowbreak a \in \allowbreak \llb (h-1)a - d, ha - d \, \rrb \cup (ha - D) \subseteq hA$ and then $x \in hA + a$, where we have especially used that $h \ge \allowbreak 1 + \allowbreak (c+d)/a$ and hence $(h-1)a - d \ge c$. In both cases, the conclusion is that $x \in hA + \{0, a\}$, which finishes the proof because $x$ is an arbitrary element in $(h+1)A$.
\end{proof}

\begin{lemma}\label{lem:2-element-sets-to-2-elements-sets}
Any automorphism of $\mathcal P_{\fin,0}(\mathbb N)$ sends $2$-element sets to $2$-element sets.
\end{lemma}

\begin{proof}
Let $\phi$ be an automorphism of $\mathcal P_{\fin,0}(\mathbb N)$ and fix $a \in \mathbb N^+$. We need to show that $B := \phi(\{0, \allowbreak a\}) = \allowbreak \{0, b\}$ for a certain $b \in \mathbb N^+$. 
For, set $n := |B| - 1$. Clearly, $n$ is a positive integer, because $0 \in B$ and $B \ne \phi(\{0\}) = \{0\}$ (by the injectivity of $\phi$). Accordingly, let $b_1, \ldots, b_n$ be an enumeration of the positive elements of $B$ with $b_n = \max B$. By Lemma \ref{lem:FTAC-decomposition}, we have that $
(k+1) B = kB + \{0, b_n\}$ for some $k \in \mathbb N$.

Put $A := \phi^{-1}(\{0, b_n\})$, where $\phi^{-1}$ is the functional inverse of $\phi$. Since $\phi^{-1}$ is an automorphism of $\mathcal P_{\fin,0}(\mathbb N)$ with $\phi^{-1}(B) = \{0, a\}$, we get from the above that
$$
(k+1)\{0, a\} = (k+1) \phi^{-1}(B) = k \phi^{-1}(B) + \phi^{-1}(\{0, b_n\}) = k \{0, a\} + A.
$$
It follows that $\{0\} \subsetneq A \subseteq (k+1) \{0, a\}$ and $\max A = (k+1) a - ka = a$. Considering that $a$ is the least non-zero element of $(k+1) \{0, a\}$, this yields $A = \{0, \allowbreak a\}$ and hence $B = \allowbreak \phi(\{0, a\}) = \phi(A) = \{0, b_n\}$.
\end{proof}

\begin{lemma}\label{lem:endos}
Let $f$ be an endomorphism of $\mathcal P_{\fin,0}(\mathbb N)$, and let $X \in \mathcal P_{\fin,0}(\mathbb N)$. The following hold:
\begin{enumerate}[label=\textup{(\roman{*})}]
\item\label{lem:endos(i)} $f(X) + k f(\{0, 1\}) = (k + \max X) f(\{0, 1\})$ for every integer $k \ge \max X$.
\item\label{lem:endos(ii)}  $\max f(X) = \max X \cdot \max f(\{0, 1\})$.
\end{enumerate}
\end{lemma}

\begin{proof}
\ref{lem:endos(i)} Let $k \in \mathbb N$. We have $\{0, \max X\} \subseteq X \subseteq \llb 0, \max X \rrb$ and $h \,\{0, 1\} = \llb 0, h \rrb$ for all $h \in \mathbb N$. It is therefore seen that, for $k \ge \max X$, 
$$
\llb 0, k + \max X \rrb = \llb 0, k \rrb \cup (\max X + \llb 0, k \rrb) \subseteq X + k \, \{0, 1\} \subseteq \llb 0, k + \max X \rrb,
$$
which ultimately shows that
$$
X + k \,\{0, 1\} = \llb 0, k + \max X \rrb = (k + \max X) \{0, 1\}.
$$
Since $f(A + B) = f(A) + f(B)$ for all $A, B \in \mathcal P_{\fin,0}(\mathbb N)$, this suffices to conclude.

\vskip 0.05cm

\ref{lem:endos(ii)} 
Using that $\max (A + B) = \max A + \max B$ for all $A, B \in \mathcal P_{\fin,0}(\mathbb N)$, we get from item \ref{lem:endos(i)} that
\begin{equation*}
\label{equ:maxima}
\max f(X) + k  \cdot \max f(\{0, 1\}) = (k + \max X) \cdot \max f(\{0, 1\})
\end{equation*}
for every integer $k \ge \max X$. Clearly, this is enough to finish the proof.
\end{proof}

Our first theorem gives a characterization of the automorphisms of $\mathcal P_{\fin,0}(\mathbb N)$ that is of independent interest and complementary to the main result (Theorem \ref{thm:main}).

\begin{theorem}\label{thm:characterization}
The following are equivalent for an endomorphism $f$ of $\mathcal P_{\fin,0}(\mathbb N)$:
\begin{enumerate}[label=\textup{(\alph{*})}]
\item\label{prop:characterization(a)} $f$ is injective and $f(\{0, 1\}) = \{0, 1\}$.
\item\label{prop:characterization(b)} $f$ is surjective.
\item\label{prop:characterization(c)} $f$ is an automorphism.
\end{enumerate}
\end{theorem}

\begin{proof}
Obviously, ``\ref{prop:characterization(a)} and \ref{prop:characterization(b)}'' $\Rightarrow$ \ref{prop:characterization(c)} $\Rightarrow$ \ref{prop:characterization(b)}. So, it suffices to show that \ref{prop:characterization(a)} $\Rightarrow$ \ref{prop:characterization(b)} $\Rightarrow$ \ref{prop:characterization(a)}. For, denote by $\mathscr S_A$ the family of all subsets $B$ of a set $A \in \allowbreak \mathcal P_{\fin,0}(\mathbb N)$ such that $\{0, \allowbreak \max A\} \subseteq \allowbreak B$. We will use without further comment that, since $\mathscr S_A$ is a finite family, a function $\mathscr S_A \to \mathscr S_A$ is injective if and only if it is surjective.

\vskip 0.05cm

\ref{prop:characterization(a)} $\Rightarrow$ \ref{prop:characterization(b)}: 
Let $Y \in \mathcal P_{\fin,0}(\mathbb N)$, and put $k := \max Y$ and $\mathscr S_k := \mathscr S_{\llb 0, k \rrb}$. We need to check that $Y = f(X)$ for some $X \in \allowbreak \mathcal P_{\fin,0}(\mathbb N)$. Considering that $f(\{0, \allowbreak 1\}) = \allowbreak \{0, 1\}$ (by hypothesis), we have from Lemma \ref{lem:endos}\ref{lem:endos(ii)} that, for every $X \in \mathscr S_k$, $\max f(X) = \max X = k$ and hence $f(X) \in \mathscr S_k$. Thus, the restriction of $f$ (which is supposed to be injective) to the family $\mathscr S_k$ results in a well-defined injection $f_k$ from $\mathscr S_k$ into it\-self. It follows that $f_k$ is bijective and, since $Y \in \mathscr S_k$, there exists $X \in \mathscr S_k$ such that $Y = \allowbreak f_k(X) = \allowbreak f(X)$.

\vskip 0.05cm

\ref{prop:characterization(b)} $\Rightarrow$ \ref{prop:characterization(a)}: Set $h := \max f(\{0, 1\}) \in \mathbb N$. By Lemma \ref{lem:endos}\ref{lem:endos(ii)}, $\max f(A) = h \max A$ for each $A \in \allowbreak \mathcal P_{\fin,0}(\mathbb N)$. So, the sur\-jec\-tivity of $f$ implies $h = \allowbreak 1$ and $f(\{0, 1\}) = \{0, 1\}$; in fact, if $h = 0$ then $f$ is identically zero, and if $h \ge 2$ then $\max f(A) \ne \allowbreak 1$ for all $A \in \allowbreak \mathcal P_{\fin,0}(\mathbb N)$. It remains to show that $f$ is also injective.

Assume $f(X) = f(Y)$ for some $X, Y \in \mathcal P_{\fin,0}(\mathbb N)$. We need to verify that $X = Y$. Set $k := \max f(X)$ and $\mathscr S_k := \mathscr S_{\llb 0, k \rrb}$. It follows from the above that $\max f(A) = \max A = k$ for all $A \in \mathscr S_k$. Consequently, the sets $X$ and $Y$ are both in $\mathscr S_k$, and the restriction of $f$ to $\mathscr S_k$ yields a well-defined surjection $f_k$ from $\mathscr S_k$ onto itself. So, similarly as in the proof of the implication \ref{prop:characterization(a)} $\Rightarrow$ \ref{prop:characterization(b)}, $f_k$ is a bijection and, since $f_k(X) = \allowbreak f(X) = f(Y) = f_k(Y)$, we conclude that $X = Y$.
\end{proof}

\begin{corollary}\label{cor:fixed-pts}
For an automorphism $f$ of $\mathcal P_{\fin,0}(\mathbb N)$, the following hold:
\begin{enumerate}[label=\textup{(\roman{*})}]
\item\label{cor:fixed-pts(iii)} $\max X = \max f(X)$ for every $X \in \mathcal P_{\fin,0}(\mathbb N)$.
\item\label{cor:fixed-pts(i)} $\{0, k\}$ and $\llb 0, k \rrb$ are fixed points of $f$ for all $k \in \mathbb N$.
\item\label{cor:fixed-pts(iv)} Either $f(\{0, 2, 3\}) = \{0, 1, 3\}$ or $f(\{0, 2, 3\}) = \{0, 2, 3\}$.
\end{enumerate}
\end{corollary}

\begin{proof}
\ref{cor:fixed-pts(iii)} This is a straightforward consequence of Lemma \ref{lem:endos}\ref{lem:endos(ii)} and Theorem \ref{thm:characterization}.

\vskip 0.05cm

\ref{cor:fixed-pts(i)} 
Let $k \in \mathbb N$. By Theorem \ref{thm:characterization}, we have $f(\{0, 1\}) = \{0, 1\}$. Since $\llb 0, k \rrb = k\, \{0, 1\}$, it is thus clear that $
f(\llb 0, k \rrb) = k f(\{0, 1\}) = \llb 0, k \rrb$; and since $\{0, k\} + \llb 0, k \rrb = \llb 0, 2k \rrb = 2\,\llb 0, k \rrb$, we find that
\begin{equation}\label{equ:image-of-2-element-set}
f(\{0, k\}) + \llb 0, k \rrb = f(\{0, k\} +\llb 0, k \rrb) = 2f(\llb 0, k \rrb) = \llb 0, 2k \rrb.
\end{equation}
In particular, we get from Eq.~\eqref{equ:image-of-2-element-set} that $
\max f( \{0, k\} ) = 2k - k = k$. So, recalling from Lemma \ref{lem:2-element-sets-to-2-elements-sets} that an automorphism of $\mathcal P_{\fin,0}(\mathbb N)$ sends $2$-elements sets to $2$-elements sets, we obtain that $f( \{0, \allowbreak k\} ) = \allowbreak \{0, k\}$.
In other words, we have proved that $\{0, k\}$ and $\llb 0, k \rrb$ are both fixed points of $f$.

\vskip 0.05cm

\ref{cor:fixed-pts(iv)} Set $Y := f(\{0, 2, 3\})$. We gather from item \ref{cor:fixed-pts(i)} that $f(\{0, 3\}) = \{0, 3\}$ and $f(\llb 0, 3 \rrb) = \llb 0, 3 \rrb$, and from item \ref{cor:fixed-pts(iii)} that $\max Y = 3$. By the injectivity of $f$, it follows that $\{0, 3\} \subsetneq Y \subsetneq \llb 0, 3 \rrb$. Therefore, the only possibility is that either $Y = \{0, 1, 3\}$ or $Y = \{0, 2, 3\}$.
\end{proof}

One consequence of Corollary \ref{cor:fixed-pts}\ref{cor:fixed-pts(iii)} is that any automorphism $f$ of $\mathcal P_{\fin,0}(\mathbb N)$ gives rise to a well-defined function $f^\ast \colon \mathcal P_{\fin,0}(\mathbb N) \to \mathcal P_{\fin,0}(\mathbb N)$, henceforth referred to as the \evid{reversal} of $f$, by taking $f^\ast(X) := \max X - \allowbreak f(X)$ for all $X \in \mathcal P_{\fin,0}(\mathbb N)$. 
The next lemma shows that something more is true.

\begin{lemma}\label{lem:reversion}
The reversal of an automorphism of $\mathcal P_{\fin,0}(\mathbb N)$ is itself an automorphism of $\mathcal P_{\fin,0}(\mathbb N)$.
\end{lemma}

\begin{proof}
Let $f$ be an automorphism of $\mathcal P_{\fin,0}(\mathbb N)$, and let $X, Y \in \mathcal P_{\fin,0}(\mathbb N)$. On account of Remark \ref{rem:autos}\ref{rem:autos(1)}, the reversal $f^\ast$ of $f$ is an endo\-mor\-phism of $\mathcal P_{\fin,0}(\mathbb N)$, as it is readily seen that
$$
f^{\ast}(X+Y) = \max(X + Y) - f(X + Y) = (\max X - f(X)) + (\max Y - f(Y)) = f^\ast(X) + f^\ast(Y).
$$
In addition, it is clear from the definition itself of $f^\ast$ that $\max f^\ast(X) = \max X$ and hence
\begin{equation}\label{equ:involutive-behavior}
f^\ast(f^\ast(X)) = \max f^\ast(X) - f^\ast(X) = \max f(X) - (\max X - f(X)) = f(X),
\end{equation}
which yields $f^\ast \circ f^\ast = f$ and shows that $f^\ast$ is a bijection (since $f$ is so). Hence, $f^\ast \in \aut(\mathcal P_{\fin,0}(\mathbb N))$.
\end{proof}

Given $X \subseteq \mathbb Z$, we denote by $\Delta(X)$ the \evid{gap set} of $X$, i.e., the set of all integers $d \ge 1$ such that $\{x, \allowbreak x + \allowbreak d\} = \allowbreak X \cap \llb x, \allowbreak x + \allowbreak d \rrb$ for some $x \in \mathbb Z$. Accordingly, we define $\maxgap(X) := \allowbreak \sup \Delta(X)$, where $\sup \emptyset := 0$.

\begin{remark}\label{rem:gaps}
It is a simple exercise (and we leave it to the reader) to show that
\begin{equation}\label{equ:11-05-23-1926}
\{0, \max X + h\} \subseteq X + \llb 0, h \rrb \subseteq \llb 0, \max X + h \rrb,
\qquad \text{for all } X \in \mathcal P_{\fin,0}(\mathbb N) \text{ and }
h \in \mathbb N.
\end{equation}
Moreover, the right-most inclusion in Eq.~\eqref{equ:11-05-23-1926} is strict if and only if $h \leq \maxgap(X) - 2$. 
\end{remark}

Note that, in the sequel, we shall often use the observations in Remark \ref{rem:gaps} without further comment.

\begin{lemma}\label{lem:max-delta-preserving}
If $f$ is an automorphism of $\mathcal P_{\fin,0}(\mathbb N)$, then $\maxgap(X) = \maxgap(f(X))$ for all $X \in \mathcal P_{\fin,0}(\mathbb N)$.
\end{lemma}

\begin{proof}
Let $X \in \mathcal P_{\fin,0}(\mathbb N)$ and put $d := \maxgap(f(X))$. Since the functional inverse of $f$ is itself an automorphism of $\mathcal P_{\fin,0}(\mathbb N)$, it suffices to prove that $\maxgap(X) \le  d$. For, set $X' := X + \llb 0, d-1 \rrb$ and suppose to the contrary that $d < \maxgap(X)$. Since $f(\{0\}) = \allowbreak \{0\}$, the gap set of $X$ is then a non-empty finite subset of $\mathbb N^+$ and hence $d$ is a \emph{positive} integer. By Remark \ref{rem:gaps} and Corollary \ref{cor:fixed-pts}\ref{cor:fixed-pts(iii)}, it follows that 
\[
f(X) + \llb 0, d-1 \rrb = \llb 0, d-1 + \max f(X) \rrb = 
 \llb 0, d-1 + \max X \rrb,
\]
which, along with Corollary \ref{cor:fixed-pts}\ref{cor:fixed-pts(i)}, implies
\[
f(X') = f(X) + f(\llb 0, d-1 \rrb) = f(X) + \llb 0, d-1 \rrb = f(\llb 0, d-1 + \max X \rrb).
\]
So, from the injectivity of $f$, we conclude that $X' = \llb 0, d-1 + \max X \rrb$. This is, however, in con\-tra\-dic\-tion with Remark \ref{rem:gaps}, as we are supposing $d < \maxgap(X)$.
\end{proof}

By Corollary \ref{cor:fixed-pts}\ref{cor:fixed-pts(iv)}, any automorphism of $\mathcal P_{\fin,0}(\mathbb N)$ belongs to one of two classes; namely, it either fixes the set $\{0, 2, 3\}$, or it maps $\{0, 2, 3\}$ to $\{0, 1, 3\}$. Our next goal is to demonstrate (in Lemma \ref{lem:fixing-the-3-sets}) that an automorphism in the first class has also to fix certain sets of the form $\{0\} \cup \llb b, c \rrb$ with $b, \allowbreak c \in \mathbb N$. 

\begin{proposition}
\label{prop:sumset-of-special-2-intervals}
For all $a, n \in \mathbb N$ with $n \ge a+1$, it holds that
\begin{equation}\label{equ:sumset-identity}
\sum_{i=0}^{n-1} \{0, a+i, a+i+1\} = \{0\} \cup \bigl\llb a, na + \tfrac{1}{2}n(n+1) \bigr\rrb.
\end{equation}
\end{proposition}

\begin{proof}
Let $A$ be the set on the left-hand side of Eq.~\eqref{equ:sumset-identity}. For each $j \in \llb 1, n \rrb$, we have
\begin{equation*}
\begin{split}
A_j & := \bigcup_{k=0}^{n-j} \, \sum_{i = k}^{k+j-1} \{a+i, a+i+1\} = \bigcup_{k=0}^{n-j} \sum_{i = 0}^{j-1} (a + k + i + \llb 0, 1 \rrb) = \bigcup_{k=0}^{n-j} \left(aj + kj + \tfrac{1}{2}(j-1)j + \llb 0, j \rrb\right) \\
& \phantom{:}= aj + \tfrac{1}{2}(j-1)j + \bigcup_{k=0}^{n-j} (kj + \llb 0, j \rrb) = aj + \tfrac{1}{2}(j-1)j +  \llb 0, (n-j+1)j \rrb;
\end{split}
\end{equation*}
in particular, $A_j$ is an interval contained in $A$, with
\begin{equation}
\label{equ:(1)-prop:sumset-of-special-2-intervals}
m_j := \min A_j = aj + \frac{1}{2} (j-1)j
\qquad\text{and}\qquad
M_j := \max A_j = m_j + (n-j+1)j. 
\end{equation}
Since $\min A^+ = a$ and $\max A = \sum_{i=0}^{n-1} (a + i+1) = na + \frac{1}{2}n(n+1)$, it follows that 
\begin{equation}
\label{equ:(2)-prop:sumset-of-special-2-intervals}
B := \{0\} \cup A_1 \cup \cdots \cup A_n \subseteq A \subseteq \{0\} \cup \left\llb a, na + \tfrac{1}{2} n(n+1) \right\rrb =: A'. 
\end{equation}
Now, using that $n \ge a+1$ (by hypothesis), we need to show that $A = A'$. If $n = 1$ (and hence $a = 0$), then it is readily checked that $A = A' = \{0, 1\}$ (and we are done). So, we may suppose $n \ge 2$. 

The function $\phi \colon \mathbb R \to \mathbb R: x \mapsto (n-x)x$ is in\-creas\-ing for $x \le \frac{1}{2}n$ and de\-creas\-ing for $x \ge \frac{1}{2}n$, so that the minimum of $\phi$ over the closed interval $[1,n-1]$ is given by $\phi(1) = \phi(n-1) = n-1 \ge \allowbreak a$.
Therefore, it is straightforward from Eq.~\eqref{equ:(1)-prop:sumset-of-special-2-intervals} that, for all $j \in \llb 1, n-1 \rrb$,
$$
m_{j+1} = a(j+1) + \tfrac{1}{2}j(j+1) = m_j + a + j \le 
m_j + (n-j)j + j = M_j.
$$
Recalling that $A_1, \ldots, A_n$ are intervals and considering that $m_1 = a$ and $M_n = an + \frac{1}{2}n(n+1)$, we thus find that
$
B = \{0\} \cup \llb m_1, M_1 \rrb \cup \cdots \cup \llb m_n, M_n \rrb = \{0\} \cup \llb m_1, M_n \rrb = A'$, which, by Eq.~\eqref{equ:(2)-prop:sumset-of-special-2-intervals}, yields $A = A'$ (as wished) and completes the proof.
\end{proof}

\begin{lemma}\label{lem:fixing-the-3-sets}
Assume $\{0, 2, 3\}$ is a fixed point of an automorphism $f$ of $\mathcal P_{\fin,0}(\mathbb N)$. The following hold:
\begin{enumerate}[label=\textup{(\roman{*})}]
\item\label{lem:fixing-the-3-sets(i)} If $1 \in X \in \mathcal P_{\fin,0}(\mathbb N)$, then $1 \in f(X)$.
\item\label{lem:fixing-the-3-sets(ii)} 
$f(\{0, a, a+1\}) = \{0, a, a+1\}$ for every $a \in \mathbb N$.
\item\label{lem:fixing-the-3-sets(iii)} $\{0\} \cup \left\llb a, na + \frac{1}{2}n(n+1) \right\rrb$ is a fixed point of $f$ for all $a, n \in \mathbb N$ with $n \ge a+1$.
\end{enumerate}
\end{lemma}

\begin{proof}
\ref{lem:fixing-the-3-sets(i)} Let $X \in \mathcal P_{\fin,0}(\mathbb N)$ and $k \in \mathbb N$. It is routine to show (by induction on $k$) that $k \{0, \allowbreak 2, 3\} = \{0\} \cup \allowbreak \llb 2, \allowbreak 3k \rrb$ (we leave the details to the reader). It is thus clear that, if $1 \in X$ and $k \ge (2 + \max X)/3$, then
$$
\llb 0, 3k + \max X \rrb = (\{0, 1\} \cup \{\max X\}) + (\{0\} \cup \llb 2, 3k \rrb) \subseteq X+k \{0, 2, 3\} \subseteq \llb 0, 3k + \max X \rrb;
$$
that is, $X+k \{0, 2, 3\}$ is the interval $\llb 0, 3k + \max X \rrb$ and hence, by Corollary \ref{cor:fixed-pts}\ref{cor:fixed-pts(i)}, a fixed point of $f$. So, by the hypothesis that $f$ is a homomorphism with $f(\{0, 2, 3\}) = \{0, 2, 3\}$, we conclude that, if $1 \in X$ and $k$ is a sufficiently large integer, then
\begin{equation}\label{equ:simple-chain-of-equals}
1 \in \llb 0, 3k + \max X \rrb = f(X + k \{0, 2, 3\}) = f(X) + k \{0, 2, 3\} = f(X) + (\{0\} \cup \llb 2, 3k \rrb).
\end{equation}
But this is only possible if $1 \in f(X)$, since the least \emph{positive} element of the set on the right-most side of Eq.~\eqref{equ:simple-chain-of-equals} is the minimum between $2$ and the least positive element of $f(X)$ (note that $\{0\} \subsetneq f(X)$).

\vskip 0.05cm

\ref{lem:fixing-the-3-sets(ii)} Let $a \in \mathbb N$, and set $X := \{0, a, a+1\}$ and $Y := f(X)$. We need to prove $Y = X$. If $a = 0$ or $a = 1$, then $X$ is an interval and we are done by Corollary \ref{cor:fixed-pts}\ref{cor:fixed-pts(i)}. So, assume $a \ge 2$. 

We gather from Corollary \ref{cor:fixed-pts}\ref{cor:fixed-pts(iii)} that $\max Y = \max X = a+1$; and from Lemma \ref{lem:max-delta-preserving} that $\delta := \maxgap(Y) = \maxgap(X) = a \ge 2$. 
It follows that $Y$ does not contain any integer $y$ in the interval $\llb 2, a-1 \rrb$, or else we would find that
$
\delta \le \allowbreak \min(y, \allowbreak a + \allowbreak 1 - \allowbreak y) \le \allowbreak a-1 < \delta$ (a contradiction). Considering that the functional inverse of $f$ is an auto\-mor\-phism of $\mathcal P_{\fin,0}(\mathbb N)$ sending $Y$ to $X$,  we infer on the other hand from item \ref{lem:fixing-the-3-sets(i)} that $1 \notin Y$. So putting everything together, we conclude from Corollary \ref{cor:fixed-pts}\ref{cor:fixed-pts(i)} that $\{0, a+1\} \subsetneq Y \subseteq \allowbreak \{0, \allowbreak a, \allowbreak a+1\} = \allowbreak X$, which shows that $Y = X$ and completes the proof.

\vskip 0.05cm

\ref{lem:fixing-the-3-sets(iii)} Given $a, n \in \mathbb N$ with $n \ge a+1$, we have from Proposition \ref{prop:sumset-of-special-2-intervals} that $\{0\} \cup \bigl\llb a, na + \frac{1}{2}n(n+1) \bigr\rrb$ can be written as the sum of the sets $\{0, a+i, \allowbreak a + \allowbreak i+1\}$ as $i$ ranges over the interval $\llb 0, n-1 \rrb$. 
By item \ref{lem:fixing-the-3-sets(ii)} and Remark \ref{rem:autos}\ref{rem:autos(3)} (specialized to $\mathcal P_{\fin,0}(\mathbb N)$), this is enough to prove the claim.
\end{proof}

\section{Main result}
\label{sect:3}

Given a set $S \subseteq \mathbb Z$, we denote by $\fdim(S)$ the smallest integer $k \ge 0$ for which there exist $k$ (discrete) intervals whose union is $S$, with the understanding that if no such $k$ exists then $\fdim(S) := \infty$. We call $\fdim(S)$ the \evid{boxing dimension} of $S$. A couple of remarks are in order.

\begin{remark}\label{rem:boxing-dim}
\begin{enumerate*}[label=\textup{(\arabic{*})}]
\item\label{rem:boxing-dim(1)} The boxing dimension of a set $S \subseteq \mathbb Z$ is zero if and only if $S$ is empty, and it is one if and only if $S$ is an interval. More generally, let us say that two sets $X, Y \subseteq \mathbb Z$ are \evid{well separated} if $|x-y| \ge 2$ for all $x \in X$ and $y \in Y$ (so, well-separated sets are disjoint). It is then easily checked that the boxing dimension of $S$ is equal to a certain integer $k \ge 0$ if and only if $S$ can be written as a union of $k$ well-separated non-empty intervals (we leave the details to the reader).\\

\indent{}For instance, the sets $A := \{0, 5\} \cup \mathbb N_{\ge 7}$ and $B := \{-2, 2, 3\}$ are well separated, with $\fdim(A) = 3$ and $\fdim(B) = 2$. To the contrary, the (set of) even integers and the odd integers are not well separated, and the boxing dimension of each of them is infinite.
\end{enumerate*}

\vskip 0.05cm

\begin{enumerate*}[label=\textup{(\arabic{*})},resume]
\item\label{rem:boxing-dim(2)} It is obvious that $\fdim(X \cup Y) \le \fdim(X) + \allowbreak \fdim(Y)$ for all $X, Y \subseteq \mathbb Z$, a property we refer to as \emph{subadditivity} (in analogy with other notions of ``dimension'' in union-closed families of sets).
\end{enumerate*}
\end{remark}

The proof of the next theorem (that is, the main theorem of the paper) is essentially an induction on the boxing dimension of a set $X \in \mathcal P_{\fin,0}(\mathbb N)$.

\begin{theorem}\label{thm:main}
The only automorphisms of the reduced power monoid $\mathcal P_{\fin,0}(\mathbb N)$ of $(\mathbb N, +)$ are the identity $X \mapsto X$ and the reversion map $X \mapsto \max X - X$.
\end{theorem}

\begin{proof}
Let $\Gamma$ be the set of automorphisms of $\mathcal P_{\fin,0}(\mathbb N)$ that fix $\{0, 2, 3\}$, and define $\Gamma^\prime := \{f^\ast \colon f \in \Gamma\}$, where $f^\ast$ denotes the reversal of $f$. We infer from Corollary \ref{cor:fixed-pts}\ref{cor:fixed-pts(iv)} and Lemma \ref{lem:reversion} that $\aut(\mathcal P_{\fin,0}(\mathbb N)) = \Gamma \cup \allowbreak \Gamma^\prime$. 
It is therefore enough to show that the only automorphism in $\Gamma$ is the identity $X \mapsto X$, whose reversal is indeed the reversion map $X \mapsto \max X - X$.

For, let $f \in \Gamma$ and $X \in \mathcal P_{\fin,0}(\mathbb N)$, and put $Y := f(X)$, $r := \fdim(X) - 1$, $s := \fdim(Y) - 1$, and $t := \allowbreak \min(r, \allowbreak s)$. We get from Corollary \ref{cor:fixed-pts}\ref{cor:fixed-pts(iii)} that $\mu := \max(X) = \max(Y)$. Moreover, $r$ and $s$ are non-negative integers, since the empty set is the only set whose boxing dimension is zero. Hence, by Remark \ref{rem:boxing-dim}, there exist increasing sequences $x_0, x_1, \ldots, x_{2r+1}$ and $y_0, y_1, \ldots, y_{2s+1}$ of integers such that 
\begin{enumerate}[label=\textup{(\Roman{*})}]
\item\label{item(1)} $x_{2i-1} + 2 \le x_{2i}$ for each $i \in \llb 1, r \rrb$ and $y_{2j-1} + 2 \le y_{2j}$ for each $j \in \llb 1, s \rrb$;

\item\label{item(2)} $X = \llb x_0, x_1 \rrb \cup \cdots \cup \llb x_{2r}, x_{2r+1} \rrb$ and $Y = \llb y_0, y_1 \rrb \cup \cdots \cup \llb y_{2s}, y_{2s+1} \rrb$. 
\end{enumerate}
In particular, $x_0 = y_0 = 0$ and $x_{2r+1} = y_{2s+1} = \mu$. We will prove by induction on $r$ that $X = Y$. 

If $r = 0$, then $X$ is an interval and we have from Corollary \ref{cor:fixed-pts}\ref{cor:fixed-pts(i)} that $X = f(X) = Y$ (as wished).
So, let $r \ge \allowbreak 1$, assume for the sake of induction that $f(S) = S$ for all $S \in \mathcal P_{\fin,0}(\mathbb N)$ with $\fdim(S) \le r$, and suppose by way of contradiction that $X \ne Y$. 
Accordingly, there is a smallest index $v \in \allowbreak \llb 1, 2t+1 \rrb$ such that $x_v \ne y_v$; otherwise, since $x_0 \le x_1 \le \cdots \le x_{2r+1}$ and $y_0 \le y_1 \le \cdots \le y_{2s+1}$, we would get from conditions \ref{item(1)} and \ref{item(2)} that 
$X = Y$, which is absurd. We distinguish two cases, depending on whether $v$ is even or odd. But first, some observations are in order.

To start with, there is no loss of generality in assuming $x_v < y_v$; otherwise, we could replace $f$ with its functional inverse $f^{-1}$, noting that $f^{-1}$ is an automorphism of $\mathcal P_{\fin,0}(\mathbb N)$ with the same fixed points as $f$. Next, it is clear that $r \le s$, or else $Y$ is fixed by $f$ (by the inductive hypothesis) and hence $f(Y) = \allowbreak Y = \allowbreak f(X)$, which is still absurd since $f$ is injective (and we have $X \ne Y$). It follows that $1 \le \allowbreak v \le \allowbreak 2r$, because $v = 2r+1$ would imply $\mu = x_{2r+1} < y_{2r+1} \le y_{2s+1} = \mu$. Lastly, we may suppose that, if $X'$ is another set in  $\mathcal P_{\fin,0}(\mathbb N)$ with $\fdim(X') = r+1$ and $X' \ne f(X')$, then $\fdim(Y) \le \allowbreak \fdim(f(X'))$; otherwise, we could trade $X$ for $X'$ and iterate the argument until the condition is satisfied (since $\fdim(Y) < \infty$, this will happen after finitely many iterations). In other words, we may assume that $X$ is \textsf{$\fdim$-cominimal} in the (non-empty) family $\mathscr S_r$ of all $S \in \mathcal P_{\fin,0}(\mathbb N)$ with $\fdim(S) = r + \allowbreak 1$ and $S \ne f(S)$, by which we mean that $\fdim(Y) \le \fdim(f(S))$ for every $S \in \mathscr S_r$.

\vskip 0.05cm

\textsc{Case 1:} $v = 2u$ for some $u \in \llb 1, r \rrb$. Put $d := x_{2u} - x_{2u-1} - 1$ and $I := \llb 0, r \rrb \setminus \{u-1, u\}$, and set 
\[
X_1 := \llb x_{2(u-1)}, x_{2u+1} + d \rrb
\qquad\text{and}\qquad
X_2 := \bigcup_{i \in I} \llb x_{2i}, x_{2i+1} + d \rrb;
\]
in particular, if $I = \emptyset$, then $X_2 = \emptyset$. It is a basic fact that
\[
\llb a, b \rrb + \llb a', b' \rrb = \llb a+a', b+b' \rrb, \qquad
\text{for all }a, b, a', b' \in \mathbb Z  \text{ with }a \le b \text{ and } a' \le b'.
\]
So, using that $X = \llb x_0, x_1 \rrb \cup \cdots \cup \llb x_{2r}, x_{2r+1} \rrb$ and $x_{2u-1} + d = x_{2u} - 1$, we have 
\begin{equation}\label{equ:sumset-as-union(1)}
X + \llb 0, d \rrb = \bigcup_{i=0}^r (\llb x_{2i}, x_{2i+1} \rrb + \llb 0, d \rrb) = \bigcup_{i=0}^r \llb x_{2i}, x_{2i+1} + d \rrb = X_1 \cup X_2.
\end{equation}
In a similar way, $Y = \llb y_0, y_1 \rrb \cup \cdots \cup \llb y_{2s}, y_{2s+1} \rrb$ leads to
\begin{equation}\label{equ:sumset-as-union(2)}
    Y + \llb 0, d \rrb = \bigcup_{j=0}^s \llb y_{2j}, y_{2j+1} + d \rrb.
\end{equation}
It follows from Eq.~\eqref{equ:sumset-as-union(1)} and Remark \ref{rem:boxing-dim}\ref{rem:boxing-dim(2)} that
\[
\fdim(X + \llb 0, d \rrb) \le \fdim(X_1) + \fdim(X_2) \le 1 + |I| = r < \fdim(X). 
\]
Consequently, we conclude from the inductive hypothesis and Corollary \ref{cor:fixed-pts}\ref{cor:fixed-pts(i)} that
$$
X + \llb 0, d \rrb = f(X + \llb 0, d \rrb) = f(X) + f(\llb 0, d \rrb) = Y + \llb 0, d \rrb. 
$$
This is however impossible, as we get from Eqs.~\eqref{equ:sumset-as-union(1)} and \eqref{equ:sumset-as-union(2)} that $x_{2u}$ is an element of $X + \llb 0, d \rrb$ but not of $Y + \allowbreak \llb 0, d \rrb$ (note that $y_{2u-1} + d  = x_{2u-1} + (x_{2u} - x_{2u-1} - 1) < x_{2u} < y_{2u}$).

\vskip 0.05cm

\textsc{Case 2:} $v = 2u+1$ for some $u \in \llb 0, r-1 \rrb$. 
Let $b$ be an integer $\ge b_0 := y_{2u+1} + \max(x_{2r}, y_{2s})$.
Since $x_0, \allowbreak \ldots, x_{2r+1}$ and $y_0, \allowbreak \ldots, \allowbreak y_{2s+1}$ are increasing sequences of non-negative integers with $0 = y_0 = x_0 < \allowbreak x_{2r} < \allowbreak x_{2r+1} = y_{2s+1} = \mu$, we have
\begin{equation}\label{equ:case2(1)}
\begin{split}
X' & := 
\bigcup_{i=0}^r \llb x_{2i} + y_{2u+1} + 1, x_{2i+1} + b \rrb \subseteq \llb y_{2u+1} + 1, \mu + b \rrb = \llb y_{2u+1} + 1, b \rrb \cup \llb b + 1, \mu + b \rrb \\
& \phantom{:}\subseteq  \llb x_0 + y_{2u+1} + 1, x_1 + b \rrb \cup  \llb x_{2r} + y_{2u+1} + 1, x_{2r+1} + b \rrb \subseteq X'
\end{split}
\end{equation}
and, in an analogous fashion,
\begin{equation}\label{equ:case2(2)}
\begin{split}
Y' & := \bigcup_{j=0}^s \llb y_{2j} + y_{2u+1} + 1, y_{2j+1} + b \rrb \subseteq \llb y_{2u+1} + 1, \mu + b \rrb = \llb y_{2u+1} + 1, b \rrb \cup \llb b + 1, \mu + b \rrb \\
& \phantom{:}\subseteq \llb y_0 + y_{2u+1} + 1, y_1 + b \rrb \cup  \llb y_{2s} + y_{2u+1} + 1, y_{2s+1} + b \rrb \subseteq Y'.
\end{split}
\end{equation}
From the above and condition \ref{item(2)} (cf.~the derivation of Eqs.~\eqref{equ:sumset-as-union(1)} and \eqref{equ:sumset-as-union(2)} in \textsc{Case 1}), it follows that
\begin{equation}\label{equ:double-equality}
X + \llb y_{2u+1} + 1, b \rrb = X' = \llb y_{2u+1} + 1, \mu + b \rrb = Y' = Y + \llb y_{2u+1} + 1, b \rrb.
\end{equation}
Let us define 
\[
a := \min(x_{2r}, y_{2u+1} + 1)
\qquad \text{and}
\qquad
Z := \{0\} \cup \llb y_{2u+1} + 1, b \rrb. 
\]
Taking $X_\ast := \llb x_0, x_1 \rrb \cup \cdots \cup \llb x_{2(r-1)}, x_{2(r-1)+1} \rrb$ and considering that 
\begin{equation}\label{equ:decomposing-X}
X = X_\ast \cup \llb x_{2r}, x_{2r+1} \rrb
\qquad\text{and}\qquad
\llb x_{2r}, x_{2r+1} \rrb \cup \llb y_{2u+1} + 1, \mu + b \rrb = \llb a, \mu + b \rrb,
\end{equation}
it is readily seen that
\begin{equation}\label{equ:X+Z}
X + Z = X \cup (X + \llb y_{2u+1} + 1, b \rrb) 
\stackrel{\eqref{equ:double-equality}}{=} X \cup \llb y_{2u+1} + 1, \mu + b \rrb \stackrel{\eqref{equ:decomposing-X}}{=} X_\ast \cup \llb a, \mu + b \rrb.
\end{equation}
Likewise, taking $Y^\ast := \llb y_{2u}, y_{2u+1} \rrb \cup \cdots \cup \llb y_{2s}, y_{2s+1} \rrb$ and 
considering that
\begin{equation}\label{equ:decomposing-Y}
Y_\ast := Y \setminus Y^\ast = \bigcup_{j=0}^{u-1} \llb y_{2j}, y_{2j+1} \rrb 
\qquad\text{and}\qquad
Y^\ast \cup \llb y_{2u+1} + 1, \mu + b \rrb = \llb y_{2u}, \mu + b \rrb,
\end{equation}
we find that
\begin{equation}\label{equ:Y+Z}
Y + Z = Y \cup (Y + \llb y_{2u+1} + 1, b \rrb) 
\stackrel{\eqref{equ:double-equality}}{=} Y \cup \llb y_{2u+1} + 1, \mu + b \rrb 
\stackrel{\eqref{equ:decomposing-Y}}{=} Y_\ast \cup \llb y_{2u}, \mu + b \rrb.
\end{equation}
In view of Remark \ref{rem:boxing-dim}\ref{rem:boxing-dim(2)}, we thus get from Eq.~\eqref{equ:X+Z} that
\begin{equation}\label{equ:fdim(X+Z)}
\fdim(X + Z) \le \fdim(X_\ast) + \fdim(\llb a, \mu+b \rrb) = r + 1 = \fdim(X),
\end{equation}
and from Eq.~\eqref{equ:Y+Z} that
\begin{equation}\label{equ:fdim(Y+Z)}
\fdim(Y + Z) \le \fdim(Y_\ast) + \fdim(\llb y_{2u}, \mu + b \rrb) = u + 1 \le (r-1) + 1 \le s < \fdim(Y).
\end{equation}
In light of Lemma \ref{lem:fixing-the-3-sets}\ref{lem:fixing-the-3-sets(iii)}, we can on the other hand choose the integer $b$ (subjected to the constraint $b \ge b_0$) in such a way that $f(Z) = Z$ and, hence,
\begin{equation}\label{equ:Z-is-fixed}
f(X+Z) = f(X) + f(Z) = Y + Z.
\end{equation}
Accordingly, we gather from Eqs.~\eqref{equ:fdim(X+Z)} and \eqref{equ:fdim(Y+Z)} that either $\fdim(X+Z) \le r$, or $\fdim(X+Z) = r+1$ and $\fdim(f(X+Z)) < \fdim(Y)$. In any case, we conclude by the inductive hypothesis and the $\fdim$-cominimality of $X$ that $X+Z$ is a fixed point of $f$, which, together with Eq.~\eqref{equ:Z-is-fixed}, implies
\[
X+Z = f(X+Z) = Y + Z.
\]
This is, however, impossible and finishes the proof. In fact, Eq.~\eqref{equ:X+Z} gives $X + Z \subseteq X \cup \allowbreak \llb a, \allowbreak b + \mu \rrb$. Since, by Condition \ref{item(1)}, $x_{2u+1} + 1 \notin X$ and $a \ge x_{2r} > x_{2u+1} + 1$, it is then clear that $x_{2u+1} + 1 \notin X+Z$. Yet, $x_{2u+1} + 1 \in Y + \allowbreak Z$, because $y_{2u} = x_{2u} < x_{2u+1}$ and hence, by Eq.~\eqref{equ:Y+Z}, $x_{2u+1} + 1 \in \llb y_{2u}, 2\mu \rrb \subseteq Y + Z$.
\end{proof}

\section{Prospects for future research}
\label{sec:conclusions}

We have found (Theorem \ref{thm:main}) that the automorphism group of the reduced power monoid $\mathcal P_{\fin,0}(\mathbb N)$ of the additive monoid $(\mathbb N, +)$ of non-negative integers has order $2$, the only non-trivial automorphism of $\mathcal P_{\fin,0}(\mathbb N)$ being the reversion map $X \mapsto \max X - X$. In fact, there is some evidence that the result can be pushed a little further.
More precisely, recall that a \evid{numerical monoid} is a submonoid $S$ of $(\mathbb N, +)$ such that $\mathbb N \setminus S$ is a finite set. We then have the following:

\begin{conjecture*}
The automorphism group of the reduced power monoid of a numerical monoid \emph{properly} contained in $\mathbb N$ is trivial (that is, the only automorphism is the identity).
\end{conjecture*}

Moreover, regardless of whether the conjecture is true or not, we are naturally led to consider an ``inverse problem'' that is probably woven much deeper in the theory:

\begin{question}
For which groups $G$ does there exist a monoid $H$ such that the automorphism group of the reduced power monoid $\mathcal P_{\fin,1}(H)$ of $H$ is isomorphic to $G$? 
\end{question}

We know from Sect.~\ref{sec:intro} that the automorphism group of the monoid $H$ embeds into the automorphism group of $\mathcal P_{\fin,1}(H)$. Yet, it is not quite clear how this can help to answer the question, as we gather from Theorem \ref{thm:main} that the latter group can be larger than the former.

\section*{Acknowledgements}

We, the authors, are thankful to the anonymous referees of an earlier version of this work for their valuable suggestions.


\begin{thebibliography}{99}
%
\bibitem{Alm02} J.~Almeida, \textit{Some Key Problems on Finite Semigroups}, Semigroup Forum \textbf{64} (2002), 159--179.

\bibitem{Got-And-2022} D.\,F.~Anderson and F.~Gotti, ``Bounded and finite factorization domains'', pp.~7--57 in: A.~Badawi and J.~Coykendall, \textit{Rings, Monoids, and Module Theory}, Springer Proc.~Math.~Stat.~\textbf{382}, Springer, 2022.

\bibitem{An-Tr18} A.\,A.~Antoniou and S.~Tringali, \textit{On the Arithmetic of Power Monoids and Sumsets in Cyclic Groups}, Pacific J.~Math.~\textbf{312} (2021), No.~2, 279--308.

\bibitem{Bie-Ger-22} P.-Y.~Bienvenu and A.~Geroldinger, \textit{On algebraic properties of power monoids of numerical monoids}, Israel J.~Math., to appear (\url{https://arxiv.org/abs/2205.00982}). 

\bibitem{Co-Tr-21(a)} L.~Cossu and S.~Tringali, \textit{Abstract Factorization Theorems with Applications to Idempotent Factorizations}, Israel J.~Math.~(2024), DOI: \url{https://doi.org/10.1007/s11856-024-2623-z}

\bibitem{Co-Tr-22(a)} L.~Cossu and S.~Tringali, \textit{Factorization under Local Finiteness Conditions}, J.~Algebra \textbf{630} (Sept.~2023), 128--161.

\bibitem{Co-Tr-22(b)} L.~Cossu and S.~Tringali, \textit{On the finiteness of certain factorization invariants}, 
Ark.~Mat.~\textbf{62} (2024), No.~1, 21--38.

\bibitem{Fa-Tr18} Y.~Fan and S.~Tringali, \textit{Power monoids: A bridge between Factorization Theory and Arithmetic Combinatorics}, J.~Algebra \textbf{512} (Oct.~2018), 252--294.

\bibitem{Ger-Hal-06} A.~Geroldinger and F.~Halter-Koch, \textit{Non-Unique Factorizations. Algebraic, Combinatorial and Analytic Theory}, Pure Appl.~Math.~\textbf{278}, Chapman \& Hall/CRC, 2006.

\bibitem{Ger-Ka-22} A.~Geroldinger and M.\,A.~Khadam, \textit{On the arithmetic of monoids of ideals}, Ark.~Mat.~\textbf{60} (2022), No.~1, 67--106.

\bibitem{Ge-Zh-20a} A.~Geroldinger and Q.~Zhong, \textit{Factorization theory in commutative monoids}, Semigroup Forum \textbf{100} (2020), 22--51.

\bibitem{Ho95} J.\,M.~Howie, \textit{Fundamentals of Semigroup Theory}, London Math.~Soc.~Monogr.~Ser.~\textbf{12}, Oxford Univ.~Press, 1995.

\bibitem{Nat1996} M.\,B.~Nathanson, \textit{Additive Number Theory: Inverse Problems and the Geometry of Sumsets}, Grad.~Texts in Math.~\textbf{165}, Springer, 1996.

\bibitem{Nat2010} M.\,B.~Nathanson, ``Addictive Number Theory'', pp.~1--8 in: D.~Chudnovsky and G.~Chudnovsky (eds.), \textit{Additive Number Theory: Festschrift In Honor of the Sixtieth Birthday of Melvyn B.~Nathanson}, Springer, 2010.
%
\bibitem{Pin1995} J.-E.~Pin, ``$BG=PG$: A success story'', pp.~33--47 in: J.~Fountain, \textit{Semigroups, Formal Languages and
Groups}, NATO ASI Ser., Ser.~C, Math.~Phys.~Sci.~\textbf{466}, Kluwer, 1995.


\bibitem{Sark08} A.~S\'ark\"ozy, \textit{On additive decompositions of the set of quadratic residues modulo $p$}, Acta Arith.~\textbf{155} (2012), No.~1, 41--51.
%
\bibitem{Tam-Sha1967} T.~Tamura and J.~Shafer, \textit{Power semigroups}, Math.~Japon.~\textbf{12} (1967), 25--32; Errata,
\textit{ibid.}~\textbf{29} (1984), No.~4, 679.
%
\bibitem{Tr21(b)} S.~Tringali, \textit{A characterisation of atomicity}, Math.~Proc.~Cam\-bridge Phil.~Soc.~\textbf{175} (2023), No.~2, 459--465.
%
\bibitem{Tr20(c)} S.~Tringali, \textit{An abstract factorization theorem and some applications}, J.~Algebra \textbf{602} (July 2022), 352--380.
%
\bibitem{Tri-Yan-23(a)} S.~Tringali and W.~Yan, \textit{A conjecture of Bienvenu and Geroldinger on power monoids}, Proc.~Amer.~Math.~Soc., to appear (\url{https://arxiv.org/abs/2310.17713}).
%
\end{thebibliography}
\end{document}